\documentclass[12pt]{article}

\usepackage{amssymb}
\usepackage{amsmath}
\usepackage{latexsym}
\usepackage{amsthm}

\usepackage{enumerate}
\usepackage{verbatim}
\usepackage{graphicx}

\usepackage{caption}
\usepackage{subcaption}

\usepackage{epstopdf}

\allowdisplaybreaks

\newcommand{\C}{\mathbb C}
\newcommand{\R}{\mathbb R}
\newcommand{\Z}{\mathbb Z}
\newcommand{\N}{\mathbb N}

\newcommand{\RR}{\mathcal R}

\newcommand{\DD}{\mathcal D}

\newcommand{\re}{\mathrm{Re}}
\newcommand{\im}{\mathrm{Im}}

\newcommand{\e}{\varepsilon}

\newcommand{\Orb}{\operatorname{Orb}}

\newtheorem{theorem}{Theorem}[section]

\newtheorem{definition}[theorem]{Definition}

\newtheorem{thm}{Theorem}[section]
\newtheorem{lem}[thm]{Lemma}
\newtheorem{prop}[thm]{Proposition}

\newtheorem{dfn}[thm]{Definition}

\newtheorem{cor}[thm]{Corollary}
\newtheorem{rk}[thm]{Remark}

\title{Hyers-Ulam stability of Hyperbolic M\"obius difference equation}

\author{Young Woo Nam}
\date{}

\begin{document}

\maketitle

\begin{abstract}
Hyers-Ulam stability of the difference equation with the initial point $ z_0 $ as follows 
$$
z_{i+1} = \frac{az_i + b}{cz_i + d}
$$
is investigated for complex numbers $ a,b,c $ and $ d $ where $ ad - bc = 1 $, $ c \neq 0 $ and $a + d \in \mathbb{R} \setminus [-2,2] $. The stability of the sequence $ \{z_n\}_{n \in \mathbb{N}_0} $ holds if the initial point is in the exterior of a certain disk of which center is $ -\frac{d}{c} $. Furthermore, the region for stability can be extended to the complement of some neighborhood of the line segment between $ -\frac{d}{c} $ and the repelling fixed point of the map $ z \mapsto \frac{az + b}{cz + d} $. This result is the generalization of Hyers-Ulam stability of Pielou logistic equation. 
\end{abstract}

\section{Introduction}
Difference equation is the recurrence relation which defines the sequence and each of which terms determines the proceeding terms. For the introduction of difference equation, for example, see \cite{{elyadi}}. The first order difference equation is of the following form
$$
z_{i+1} = g(i,z_i)
$$
for all integer $ i \geq 0 $. %
In 1940, Ulam \cite{ulam} suggested the problem concerning
the stability of group homomorphisms \cite{PoADE}:
Given a metric group $(G, \cdot, d)$, a
positive number $\varepsilon$, and a function $f : G \to G$
which satisfies the inequality
$d \big( f(xy),\; f(x)f(y) \big) \leq \varepsilon$ for all
$x, y \in G$, do there exist an homomorphism $a : G \to G$ and
a constant $\delta$ depending only on $G$ and $\varepsilon$
such that $d \big( a(x), f(x) \big) \leq \delta$ for all
$x \in G$?
A first answer to this question was given by Hyers \cite{hyers}
in 1941 who proved that the Cauchy additive equation is stable
in Banach spaces. 
%
\smallskip \\
The difference equation has Hyers-Ulam stability if each terms of the sequence with the given relation has (small) error, this sequence is approximated by the sequence with same relation which has no error. Hyers-Ulam stability of difference equation is relatively recent topic. For example, see \cite{jung1, jungnam, jungnam1, PoADE}. 
In particular, Pielou logistic difference equation has Hyers-Ulam stability only if the initial point of the sequence is contained in definite intervals in \cite{jungnam1}. In the same paper, this result is extended to the following difference equation over $ \R $
\begin{align*}
x_{i+1} = \frac{ax_i + b}{cx_i + d}
\end{align*}
where $ ad - bc = 1 $, $ c \neq 0 $ and $ ( a + d )^2 > 4 $ for real numbers $ a $, $ b $, $ c $ and $ d $. In this article, we generalize the result of difference equations on the complex plane $ \C $ where $ a + d $ is real and it satisfies that $ ( a + d )^2 > 4 $ with complex numbers $ a $, $ b $, $ c $ and $ d $. 

\subsection*{M\"obius map}
Linear fractional map on the Riemann sphere $ \hat{\C} = \C \cup \{\infty\} $ is called {\em M\"obius map} or {\em M\"obius transformation}. 
$$ g(z) = \dfrac{az + b}{cz + d} $$ 
where $ ad - bc \neq 0 $ for $ z \in \hat{\C} $.
\smallskip \\
The non-constant M\"obius map $ g(z) = \frac{az+b}{cz+d} $ has the following properties.
\begin{itemize}
\item Without loss of generality, we may assume that $ ad - bc =1 $.
\item $ g(\infty) $ is defined as $ \frac{a}{c} $ and $ g \left(-\frac{d}{c} \right) $ is defined as $ \infty $.
\item The composition of two M\"obius maps is also a M\"obius map.
\item The map $ g $ is the linear map if and only if $ \infty $ is a fixed point of $ g $.
\item The image of circle or line under M\"obius map is circle or line.
\end{itemize}
The matrix representation of M\"obius map is useful to classify M\"obius map qualitatively. In particular, the equation $ \frac{az+b}{cz+d} = \frac{paz + pb}{pcz +pd} $ holds for all $ p \neq 0 $. We define the 
%
matrix representation of M\"obius map $ z \mapsto \frac{az+b}{cz+d} $ as follows
%
%
$ \left( 
\begin{smallmatrix}
a & b \\
c & d
\end{smallmatrix}
\right) $
where $ ad - bc = 1 $. We denote the matrix representation of M\"obius map $ g $ by also $ g $ unless it makes confusion. 
Denote the trace of the matrix representation of M\"obius map $ g $ by $ \mathrm{tr}(g) $. 
%
\subsection*{Main content}
 In Section 3, Hyers-Ulam stability of the sequence defined by hyperbolic M\"obius map on the exterior of the disk of which center is $ g^{-1}(\infty) $ with a certain radius. This is the direct generalization of Hyers-Ulam stability of Pielou logistic equation in \cite{jungnam1} on the complex plane. In Section 5, the {\em avoided region} at $ \infty $ is defined as the complement of the closure of the neighborhood of the line segment between $ g^{-1}(\infty) $ and the repelling fixed point of $ g $. 
 In Section 7, Hyers-Ulam stability of $ g $ is proved in the complement of an avoided region. 

\smallskip

\section{Hyperbolic M\"obius map}
 
The trace of matrix is invariant under conjugation. Thus qualitative classification of M\"obius map depends on the trace of matrix representation.
\begin{dfn}
If the matrix representation of the non-constant M\"obius map  
$ \left( 
\begin{smallmatrix}
a & b \\
c & d
\end{smallmatrix}
\right) $
has its trace $ a+d $, say $ \mathrm{tr}(g) $, is in the set $ \R \setminus [-2, 2] $, then the map $ g $ is called the {\em hyperbolic} M\"obius map.
\end{dfn}
Denote the fixed points of $ g $ by $ \alpha $ and $ \beta $. If $ |g'(\alpha)| < 1 $, then $ \alpha $ is called the {\em attracting} fixed point. If $ |g'(\beta)| > 1 $, then $ \beta $ is called the {\em repelling} fixed point.


\begin{lem} \label{lem-fixed points of hyp Mobius map}
Let $ g $ be the hyperbolic M\"obius map such that $ g(z) = \frac{az +b}{cz +d} $ where $ ad - bc =1 $ and $ c \neq 0 $. Then $ g $ has two different fixed points, one of which is the attracting fixed point and the other is the repelling fixed point.
\end{lem}

\begin{proof}
The fixed points of $ g $ are the roots of the quadratic equation
$$ cz^2 - (a-d)z - b =0 $$
Denote the fixed points of $ g $ as follows
\begin{align} \label{eq-fixed points of g}
\alpha = \frac{a-d + \sqrt{(a+d)^2 - 4}}{2c} \quad \textrm{and} \quad \beta = \frac{a-d - \sqrt{(a+d)^2 - 4}}{2c}. 
\end{align}
Observe that $ \alpha + \beta = \frac{a-d}{c} $ and $ \alpha \beta = -\frac{b}{c} $. Thus we have the following equation
\begin{align} \label{eq-ad-bc is 1}
 \nonumber
(c\alpha +d)(c\beta +d) &= c^2\alpha\beta + cd (\alpha + \beta ) + d^2 \\ \nonumber
&= -bc + d(a-d) +d^2 \\ \nonumber
&= -bc +ad \\
&=1
\end{align}
Since $ g $ is the hyperbolic M\"obius map, that is, $ a+d > 2 $ or $ a+d < -2 $, without loss of generality we may assume that $ a+d >2 $. Then we obtain the following inequality using the equation \eqref{eq-fixed points of g}
\begin{align} \label{eq-c alpha plus d}
c\alpha + d = \frac{a+d + \sqrt{(a+d)^2 - 4}}{2} > \frac{a+d}{2} > 1.
\end{align}
Since $ g'(z) = \frac{1}{(cz+d)^2} $ and by the equations \eqref{eq-ad-bc is 1} and \eqref{eq-c alpha plus d}, we obtain that $ g'(\alpha) = \frac{1}{(c\alpha+d)^2} < 1 $ and $ g'(\beta) = \frac{1}{(c\beta+d)^2} > 1 $.
\end{proof}

\smallskip

\begin{lem} \label{lem-congugation of hyp Mobius map}
Let $ g $ and $ h $ are M\"obius map as follows
$$
g(z) = \frac{az+b}{cz+d} \quad  \text{and} \quad h(z) = \frac{z-\beta}{z-\alpha}
$$
where $ \alpha $ and $ \beta $ are the fixed points of $ g $ and $ ad-bc =1 $. If $ \alpha \neq \beta $, then $ h \circ g \circ h^{-1} (w) = kw $ where $ k = \frac{1}{(c\beta +d)^2} $. In particular, if $ g $ is the hyperbolic M\"obius map and $ \beta $ is the repelling fixed point, then $ k>1 $. 
\end{lem}

\begin{proof}
The maps $ g $ and $ h $ are M\"obius map. Thus so is $ h \circ g \circ h^{-1} $. By the direct calculation, we obtain that $ h^{-1}(w) = \frac{\alpha w - \beta}{w - 1} $.  Observe that $ h^{-1}(0) = \beta $, $ h^{-1}(\infty) = \alpha $ and $ h^{-1}(1) = \infty $. Then we have
\begin{align*}
h \circ g \circ h^{-1}(0) &= h \circ g (\beta) = h(\beta) = 0 \\
h \circ g \circ h^{-1}(\infty) &= h \circ g (\alpha) = h(\alpha) = \infty
\end{align*}
The points $ 0 $ and $ \infty $ are fixed points of $ h \circ g \circ h^{-1} $. So $ h \circ g \circ h^{-1} (w) = kw $ for some $ k \in \C $. Since $ k = h \circ g \circ h^{-1} (1) $, the following equation holds by \eqref{eq-fixed points of g} and \eqref{eq-ad-bc is 1}
\begin{align*}
k &= h \circ g \circ h^{-1} (1) = h \circ g(\infty) = h \left( \frac{a}{c} \right) & \\
&= \frac{\frac{a}{c} - \beta}{\frac{a}{c} - \alpha} = \frac{a - c\beta}{a - c\alpha} & 
\\
&= \frac{a + d + \sqrt{(a+d)^2 - 4}}{a + d - \sqrt{(a+d)^2 - 4}}  
\\
&= \frac{c\alpha +d}{c\beta + d} 
\\
&= \frac{1}{(c\beta +d)^2}.  
\end{align*}
%
If $ g $ is the hyperbolic M\"obius map, then $ k = \frac{1}{(c\beta +d)^2} = g'(\beta) > 1 $ by the proof of Lemma \ref{lem-fixed points of hyp Mobius map}. 
\end{proof}

\smallskip

\begin{lem} \label{lem-hyperbolic Mobius transformation}
Let $ g $ be the hyperbolic M\"obius map on $ \hat{\C} $. Let $ \alpha $ and $ \beta $ be the attracting and the repelling fixed point respectively. Then 
$$
\lim_{n \rightarrow \infty} g^n(z) \rightarrow \alpha \quad \text{as} \quad n \rightarrow +\infty
$$ 
for all $ z \in \hat{\C} \setminus \{\beta \} $. 
\end{lem}

\begin{proof}
By the classification of M\"obius map, the hyperbolic M\"obius map has both the attracting and the repelling fixed points. Let $ h $ be the linear fractional map as follows
$$
h(z) = \frac{z - \beta}{z - \alpha} .
$$
Then $ f = h \circ g \circ h^{-1} $ is the dilation with the repelling fixed point at zero, that is, $ f(w) = kw $ for $ k > 1 $. Thus  $ 0 $ is the repelling fixed point of $ f $. Since $ h $ is a bijection on $ \hat{\C} $, the orbit, $ \{ g^n(z) \}_{n \in \Z} $ corresponds to the orbit, $ \{ f^n(h(z)) \}_{n \in \Z} $ by conjugation $ h $. Observe that 
$$
f^n(z) \rightarrow \infty \quad \text{as} \quad n \rightarrow +\infty
$$
for all $ z \in \hat{\C} \setminus \{ 0 \} $. Hence, 
$$
g^n(z) \rightarrow \alpha \quad \text{as} \quad n \rightarrow +\infty
$$
for all $ z \in \hat{\C} \setminus \{ \beta \} $. 
\end{proof}

\smallskip

\begin{cor}
Let $ g $ be the map defined in Lemma \ref{lem-hyperbolic Mobius transformation}. Then 
$$
\lim_{n \rightarrow \infty} g^{-n}(z) \rightarrow \beta \quad \text{as} \quad n \rightarrow +\infty
$$ 
for all $ z \in \hat{\C} \setminus \{\alpha \} $.
\end{cor}

\begin{proof}
Observe that $ g^{-1} $ is also hyperbolic M\"obius transformation and $ \beta $ and $ \alpha $ are the attracting and the repelling fixed point under $ g^{-1} $ respectively. Thus we apply the proof of Lemma \ref{lem-hyperbolic Mobius transformation} to the map $ g^{-1} $. It completes the proof.  
\end{proof}
\smallskip
We collect the notions throughout this paper as follows
\begin{itemize}
\item The M\"obius map $ g $ is the hyperbolic M\"obius map  and $ g(z) = \frac{az + b}{cz +d} $ where $ ad -bc =1 $ and $ c \neq 0 $.
\item The M\"obius map $ h $ is defined as $ h(z) = \frac{z - \beta}{z - \alpha} $ where $ \alpha $ and $ \beta $ are the attracting and the repelling fixed point of $ g $.
\item Without loss of generality, we may assume that the hyperbolic M\"obius map $ g $ has the matrix representation with $ \mathrm{tr}(g) > 2 $. 
\item Since the trace of the matrix is invariant under conjugation, we obtain that $ \mathrm{tr}(g) = \mathrm{tr}(h \circ g \circ h^{-1}) $. By Lemma \ref{lem-congugation of hyp Mobius map}, if $ \mathrm{tr}(g) >2 $, then  
\begin{align*}
\mathrm{tr}(g) = \mathrm{tr}
\begin{pmatrix}
\sqrt{k} & 0 \\ 
0 & \frac{1}{\sqrt{k}}
\end{pmatrix} 
 = \sqrt{k} + \frac{1}{\sqrt{k}} > 2 .
\end{align*}

\end{itemize}

\smallskip

\section{Hyers-Ulam stability on the exterior of disk}




%
%
Let $ F $ be the function from $ \N_0 \times \C $ to $ \C $. Suppose that for a given positive number $ \e $, a complex valued sequence $ \{ a_n \}_{n \in \N_0 } $ satisfies the inequality
$$
| a_{i+1} - F(i,a_i) | \leq \e
$$
for all $ i \in \N_0 $ where $ | \cdot | $ is the absolute value of complex number. If there exists the sequence $ \{ b_n \}_{n \in \N_0 } $ which satisfies that 
$$ b_{i+1} = F(i,b_i) $$
for each $ i \in \N_0 $,
and $ |a_i - b_i | \leq G(\e) $ for all $ i \in \N_0 $ where the positive number $ G(\e) $ converges to zero as $ \e \rightarrow 0 $,  then we say that the sequence $ \{ b_n \}_{n \in \N_0 } $ has {\em Hyers-Ulam stability}. Denote $ F(i, z) $ by $ F_i(z) $ if necessary. 
%
%

\medskip

%
%
%
%
%
%
%
%
%
%

The set $ S $ is called an invariant set under $ F $ (or $ S $ is invariant under $ F $) where for any $ s \in S $ we obtain that $ F(i,s) \in S $ for all $ i \in \N_0 $. 

\begin{lem} \label{lem-hyers ulam stability with contraction}
Let $ F \colon \N_0 \times \C \rightarrow \C $ be a function satisfying the condition
\begin{align} \label{eq-condition of F}
|F(i,u) - F(i,v)| \leq K|u-v|
\end{align}
for all $ i \in \N_0 $, $ u,v \in \C $ and for $ 0 < K < 1 $. For a given an $ \e > 0 $ suppose that the complex valued sequence $ \{a_i \}_{i \in \N_0} $ satisfies the inequality 
\begin{align}  \label{eq-sequence a-n}
|a_{i+1} - F(i,a_i)| \leq \e 
\end{align}
for all $ i \in \N_0 $. Then there exists a sequence $ \{b_i \}_{i \in \N_0} $ satisfying 
\begin{align}   \label{eq-sequence b-n}
b_{i+1} = F(i,b_i) 
\end{align}
and
\begin{align*}
|b_i - a_i| \leq K^i|b_0 - a_0| + \frac{1-K^i}{1-K} \,\e
\end{align*}
for $ i \in \N_0 $. If the whole sequence $ \{a_i \}_{i \in \N_0} $ is contained in the invariant set $ S \subset \C $ under $ F $ 
, then $ \{b_i \}_{i \in \N_0} $ is also in $ S $ under the condition, $ a_0 = b_0 $.
\end{lem}

\begin{proof}
By induction suppose that 
$$
| b_{i-1} - a_{i-1} | \leq K^{i-1} |b_0 -a_0| + \frac{1-K^{i-1}}{1-K} \,\e .
$$
If $ i=0 $, then trivially $ |b_0 - a_0| \leq \e $. Morover,
\begin{align*}
|b_i - a_i| & \leq | b_i - F(i-1,a_{i-1})| + | a_i - F(i-1,a_{i-1})| \\[0.2em]
 & \leq | F(i-1,b_{i-1}) - F(i-1,a_{i-1})| + | a_i - F(i-1,a_{i-1})| \\[0.2em]
 & = K | b_{i-1} - a_{i-1} | + \e \\
 & \leq K \left\{ K^{i-1} |b_0 -a_0| + \frac{1-K^{i-1}}{1-K} \,\e \right\} + \e \\
 &= K^{i} |b_0 -a_0| + \frac{1-K^{i}}{1-K} \,\e .
\end{align*}
Moreover, if $ a_0 = b_0 $, then the sequence $ \{b_i \}_{i \in \N_0} $ satisfies the inequality \eqref{eq-sequence a-n} without error under $ F $. Hence, $ \{b_i \}_{i \in \N_0} $ is contained in the invariant set $ S $. 
\end{proof}


%
%
%
%



The real version of the following lemma is proved in \cite{jungnam1} as $ g $ is the map defined on the real line. 
%
%

%
\begin{figure}
    \centering
    \includegraphics[scale=0.85]{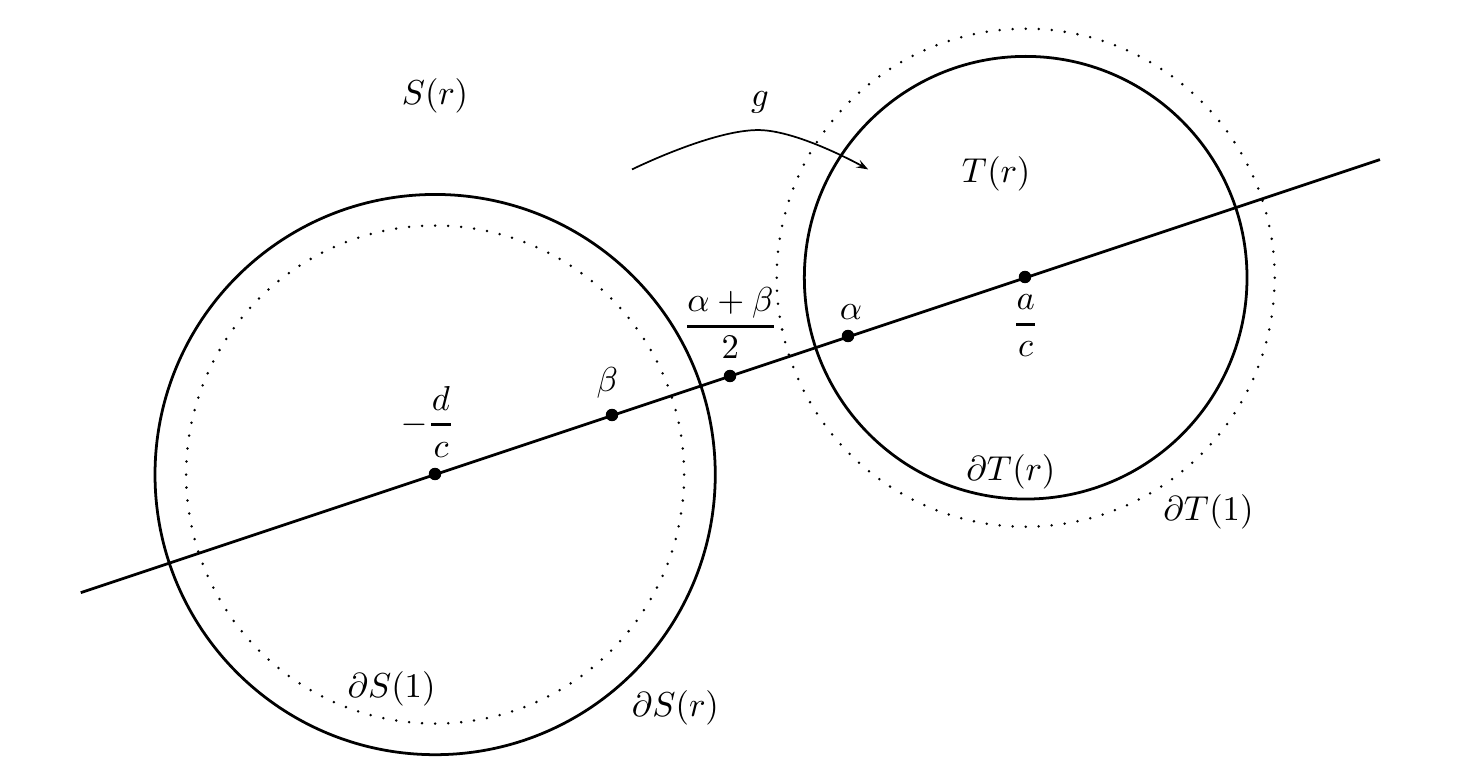}   
    \caption{Image of the disk under hyperbolic M\"obius map}
\end{figure}

\begin{lem} \label{lem-invariant region of Mobius}
Let $ g $ be the M\"obius map $ g(z) = \frac{az+b}{cz+d} $
where $ a $, $ b $, $ c $ and $ d $ are complex numbers, $ ad - bc = 1 $ and $ c \neq 0 $. Let the region $ S(r) $ and $ T(r) $ be as follows
\begin{align*}
S(r) = \left\{ z \in \C \colon \left| z + \frac{d}{c} \right| > \frac{r}{|c|} \right\}, \quad
T(r) = \left\{ z \in \C \colon \left| z - \frac{a}{c} \right| < \frac{1}{r |c|} \right\} 
\end{align*}
for $ r > 0 $. Then $ g (S(r)) = T(r)\setminus \left\{ \dfrac{a}{c}  \right\} $ for any $ r > 0 $. Moreover, if $ g $ is hyperbolic M\"obius map and $ r + \frac{1}{r} < | \mathrm{tr}(g) | $, then the closure of $ T(r) $ is contained in $ S(r) $. 
\end{lem}

\begin{proof}
The set $ S(r) $ is contained in $ \C $ and $ g(\infty) = \frac{a}{c} $. Then $ S(r) $ does not have $ \frac{a}{c} $. 
The equation $ g (S(r)) = T(r)\setminus \left\{ \dfrac{a}{c}  \right\} $ for any $ r > 0 $ is shown by the following equivalent conditions
\begin{align*}
 g(z) \in T(r)\setminus \left\{ \dfrac{a}{c}  \right\} & \Longleftrightarrow 0 < \left| \frac{az+b}{cz+d} - \frac{a}{c} \right| < \frac{1}{r|c|}  \\
 & \Longleftrightarrow 0 < \left| \frac{ad - bc}{cz^2 + cd} \right| < \frac{1}{r|c|} \\
 & \Longleftrightarrow 0 < \frac{1}{\left| z + \frac{d}{c} \right|} < \frac{|c|}{r} \\ 
 & \Longleftrightarrow \left| z + \frac{d}{c} \right| > \frac{r}{|c|} \\[0.2em]
 & \Longleftrightarrow \ z \in S(r) .
\end{align*}
Additionally, suppose that $ g $ is the hyperbolic map and $ r + \frac{1}{r} < | \mathrm{tr}(g)| $. The closure of $ T(r) $ is the set of points satisfying that $ \left| z - \frac{a}{c} \right| \leq \frac{1}{r|c|} $. Then for any $ z $ in the closure of $ T(r) $, we obtain that 
\begin{align*}
| cz + d| &= | cz -a + a+d | \\
& \geq - |cz -a| + |a+d| \\
& = -|c| \cdot \left| z - \frac{a}{c} \right| + | \mathrm{tr}(g) | \\
& > -|c| \cdot \frac{r}{|c|} + r + \frac{1}{r} \\
& = \frac{1}{r}
\end{align*}
Then we have $ \left| z + \frac{d}{c} \right| > \frac{1}{r|c|} $, that is, $ z \in S(r) $. Hence, the closure of $ T(r) $ is contained in $ S(r) $ if $ r + \frac{1}{r} < | \mathrm{tr}(g) | $.
%
%
%
%
%
\end{proof}

%
\begin{prop} \label{prop-stability on S1+t}
Let $ g $ be the hyperbolic M\"obius map with $ \mathrm{tr}(g) = 2 + \tau $ for $ \tau > 0 $. Let $ S(r) $ be the region defined in Lemma \ref{lem-invariant region of Mobius}. For a given $ \e > 0 $, let a complex valued sequence $ \{ a_i \}_{i \in \N_0 } $ satisfies the inequality
$$
| a_{i+1} - g(a_i) | \leq \e
$$
for all $ i \in \N_0 $. 
 Suppose that $ \e < \frac{t}{|c|(1 + t)} $ and $ a_0 $ is in $ S(1+t) $ for $ 0 < t \leq \tau $. Then the sequence $ \{ a_n \}_{n \in \N_0} $ is contained in $ S(1+t) $. Moreover, there exists the sequence $ \{ b_i \}_{i \in \N_0 } $ satisfying 
$$ b_{i+1} = g(b_i) $$
for each $ i \in \N $ has Hyers-Ulam stability where $ b_0 = a_0 $. 
\end{prop}

\begin{proof}
For the map 
$$ g(z) = \frac{az+b}{cz+d} $$
we may assume that $ ad-bc$ $  =1 $. Recall that $ g'(z) = \frac{1}{(cz+d)^2} $. Thus $ |g'| $ has a uniform upper bound in $ S(1+t) $ as follows 
\begin{align} \label{eq-contraction in S under hyp}
z \in S(1 + t) & \Longleftrightarrow \left| z + \frac{d}{c} \right| > \frac{1+t}{|c|} \nonumber \\
& \Longleftrightarrow |cz + d| > 1 + t \nonumber \\
& \Longleftrightarrow |g'(z)| = \frac{1}{|cz+d|^2} < \frac{1}{(1+ \tau)^2} < 1 .
\end{align}
\textit{claim} $ \colon $ If $ a_0 \in S(1+\tau) $ and $ \e < \frac{\tau}{|c|(1 + \tau)} $, then the whole sequence $ \{ a_i \}_{i \in \N_0 } $ is also contained in $ S(1+\tau) $. By induction assume that $ a_{i-1} \in S(1+\tau) $. Then 
\begin{align*}
\left| a_i - \left( -\frac{d}{c} \right) \right| &= \left| a_i - \frac{a}{c} + \frac{a}{c} -\frac{d}{c} \right| \\
& \geq - \left| a_i - \frac{a}{c} \right| + \left| \frac{a+d}{c} \right| \\
& \geq - \left| g(a_{i-1}) - \frac{a}{c} \right| - \e + \frac{2+t}{|c|} \\
& > - \frac{1}{|c|(1 + \tau)} - \e + \frac{2+t}{|c|} \qquad \qquad \textrm{by Lemma \ref{lem-invariant region of Mobius}} \\
& = \frac{1}{|c|} \left( 2 + t - \frac{1}{1+t} - \e \right) \\
& > \frac{1}{|c|} \left( 2 + t - \frac{1}{1+t} - \frac{t}{1 + t} \right) \\
& = \frac{1}{|c|} (1 + t) .
\end{align*}
Thus $ \left| a_i +\frac{d}{c} \right| > \frac{1+t}{|c|} $, that is, $ a_i \in S(1 + t) $. Then the whole sequence $ \{ a_i \}_{i \in \N_0 } $ is contained in $ S(1+\tau) $. \\
The inequality \eqref{eq-contraction in S under hyp} implies that $ g $ is the Lipschitz map with Lipschitz constant $ \frac{1}{(1+t)^2} $ in $ S(1 + t) $. Then Lemma \ref{lem-hyers ulam stability with contraction} implies that 
%
%
%
%
\begin{align*}
|b_i - a_i| \leq K^i|b_0 - a_0| + \frac{1 + K^i}{1 - K} \,\e
\end{align*}
where $ K = \dfrac{1}{(1+t)^2} < 1 $. Hence, the sequence $ \{ b_i \}_{i \in \N_0} $ has Hyers-Ulam stability where $ b_0 = a_0 $. 
\end{proof}

\smallskip

\section{Image of concentric circles under the conjugation $h$}
In this section, we show that the image of concentric circles of which center is $ -\frac{d}{c} $ under the map $ h $ defined as $ h(z) = \frac{z- \beta}{z- \alpha} $. Denote a circle in the complex plane by $ C $. Recall that the image of line or circle under M\"obius map is line or circle. Moreover, since M\"obius map is conformal, the end points of the diameter of $ C $ is mapped by $ h $ to the end points of the diameter of $ h(C) $. However, the image of the center of $ C $ is not in general the center of $ h(C) $. 
\begin{figure}
    \centering
    \includegraphics[scale=1]{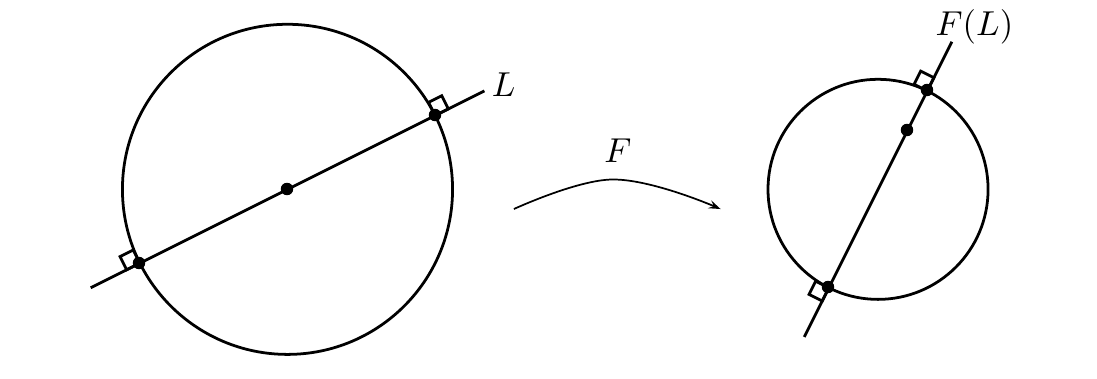}   
    \caption{Circle and line under M\"obius map}
\end{figure}
Recall the map $ f = h \circ g \circ h^{-1} $ is the dilation defined as $ f(w) = kw $ where $ k = \frac{1}{(c \beta + d)^2} > 1 $ by Lemma \ref{lem-congugation of hyp Mobius map}. 
%
%
%
%
%
Let $ L $ be the straight line in $ \C $. Define the {\em extended (straight) line} 
as $ L \cup \{ \infty \} $ and denote it by $ L_{\infty} $.
\smallskip
\begin{lem} \label{lem-image of -d/c under h}
Let $ h $ be the M\"obius map defined as $ h(z) = \frac{z- \beta}{z- \alpha} $. Then the image of $ -\frac{d}{c} $ under $ h $ as follows 
$$ h\left(-\frac{d}{c} \right) = \frac{1}{k} \ \  \textrm{and} \ \ -\frac{d}{c} = \frac{k \beta - \alpha}{k-1} . $$ 
\end{lem}

\begin{proof}
The map $ h $ is the conjugation from $ g $ to $ f $ and $ h(\infty) = 1 $. The fact that $ f \circ h = h \circ g $ implies that
$$
f \circ h\left(-\frac{d}{c} \right) = h \circ g \left(-\frac{d}{c} \right) = h(\infty) = 1 = f \left(\frac{1}{k} \right) . 
$$
Since $ f $ is a bijection on $ \C $, $ h(-\frac{d}{c}) = \frac{1}{k} $. Observe that the map $ h^{-1}(w) = \frac{\alpha w - \beta}{w - 1} $. Hence, we have 
\begin{align} \label{eq-inverse image of -d/c}
-\dfrac{d}{c} = h^{-1}\left( \dfrac{1}{k} \right) = \dfrac{k \beta - \alpha}{k-1} .
\end{align}
\end{proof}

\begin{lem} \label{lem-invariance of extended line of g}
Let $ g $ be the hyperbolic M\"obius map $ g(z) = \frac{az + b}{cz + d} $ where $ ad-bc=1 $ and $ c \neq 0 $ with the attracting and repelling fixed points $ \alpha $ and $ \beta $ respectively. Denote the extended (straight) line which contains $ \alpha $ and $ \beta $ by  $ L_{\infty} = \{ \;t\alpha + (1-t)\beta \; \colon \;t \in \R \;\} \cup \{ \infty \} $. Then $ L_{\infty} $ is invariant under $ g $. In particular, $ g(L_{\infty}) = L_{\infty} $. 
\end{lem}

\begin{proof}
Lemma \ref{lem-image of -d/c under h} implies the following equation 
$$
- \frac{d}{c} = \frac{k}{k-1}\beta - \frac{1}{k-1} \alpha
$$
Thus $ -\frac{d}{c} $ is in the line segment which contains both $ \alpha $ and $ \beta $. Then $ -\frac{d}{c} \in L_{\infty} $. The fact that $ g\left(- \frac{d}{c} \right) = \infty $ implies that $ \infty $ is also contained in $ g(L_{\infty}) $. Recall that 
any non constant M\"obius map is bijective on $ \hat{\C} $.  Then $ g(L_{\infty}) $ is the extended line which contains $ \alpha $ and $ \beta $. Hence, the extended line $ L_{\infty} $ is invariant under $ g $ and $ g(L_{\infty}) = L_{\infty} $. 
\end{proof}

\begin{cor}
Let $ L_{\infty} $ be the extended line defined in Lemma \ref{lem-invariance of extended line of g}. Let $ h $ be the map defined as $ \frac{z-\beta}{z-\alpha} $ where $ \alpha $ and $ \beta $ are the attracting and the repelling fixed points of $ g $ respectively. Then $ h(L_{\infty}) $ is the extended real line, $ \R \cup \{ \infty \} $. 
\end{cor}

\begin{proof}
By the definition of $ h $, $ h(\alpha) = \infty $, $ h(\beta) = 0 $ and $ h(\infty) = 1 $. Hence, $ h(L_{\infty}) $ is the extended line which contains $ 0, 1 $ and $ \infty $. Hence, $ h(L_{\infty}) $ is the extended real line, $ \R \cup \{ \infty \} $. 
\end{proof}
%
%
%
%
%
Recall the set $ \partial S(r) $ is the circle of which center is $ -\frac{d}{c} $ with radius $ \frac{r}{|c|} $ for $ c \neq 0 $. The extended line $ L_{\infty} $ contains $ -\frac{d}{c} $. Thus for any two points $ p $ and $ q $ in $ L_{\infty} $, if the midpoint of the line segment connecting $ p $ and $ q $ is $ -\frac{d}{c} $, then $ p $ and $ q $ are the end points of the diameter of $ \partial S(r) $ for some $ r > 0 $. In the following lemma the image of the endpoints of the diameter of $ \partial S(r) $ for arbitrary radius. 
\begin{lem} \label{lem-image of the points in the line alpha and beta}
Let $ h $ be the map $ h(z) = \frac{z-\beta}{z - \alpha} $. 
Let $ t\alpha + (1-t)\beta $ be the point in $ L_{\infty} $ and denote it by $ p_t $ for $ t \in \R $. 
Then 
\begin{align*}
h(p_t) = \frac{t}{t-1} \quad \ \text{and} \quad \ h \left(-p_t - \frac{2d}{c} \right) = \frac{tk -t + 2}{tk - t + k +1} .
\end{align*}
\end{lem} 

\begin{proof}
The straightforward calculation shows that $ h(p_t) = \frac{t}{t-1} $. So we omit the detail of the calculation. 
%
%
%
By the equation \eqref{eq-inverse image of -d/c} in Lemma \ref{lem-image of -d/c under h}, we have $ -\frac{d}{c} = \frac{k \beta - \alpha}{k-1} $. Then
\begin{align*}
h \left(-p_t- \frac{2d}{c} \right) 
&= \frac{\left(-p_t- \frac{2d}{c} \right) - \beta}{\left(-p_t- \frac{2d}{c} \right) - \alpha} \\
&= \frac{-t\alpha - (1-t)\beta - \beta - \frac{2d}{c}}{-t\alpha - (1-t)\beta - \alpha - \frac{2d}{c}} \\[0.2em]
&= \frac{t\alpha + (2-t)\beta - \frac{2k\beta - 2 \alpha}{k-1}}{(1+t) \alpha + (1-t) \beta - \frac{2k\beta - 2 \alpha}{k-1}} \qquad \textrm{by \eqref{eq-inverse image of -d/c} in Lemma \ref{lem-image of -d/c under h}}\\[0.2em]
&= \frac{(tk - t + 2)(\alpha - \beta)}{(tk - t + k +1)(\alpha - \beta)} \\[0.2em]
&= \frac{tk - t + 2}{tk - t + k +1} .
\end{align*}
It completes the proof.
\end{proof}

\smallskip
\begin{cor} \label{cor-image under h at t0 and t half}
Let $ h $ be the M\"obius map $ h(z) = \frac{z- \beta}{z - \alpha} $. Then we obtain that 
%
\begin{align*}
h\left(-\beta - \frac{2d}{c} \right) = \frac{2}{k+1}, \quad 
&\text{and} \quad  h \left(-\frac{\alpha + \beta}{2} - \frac{2d}{c} \right) = \frac{k + 3}{3k + 1} .
\end{align*}
\end{cor}

\begin{proof}
Observe that $ p_0 = \beta $ and $ p_{\frac{1}{2}} = \frac{\alpha + \beta}{2} $ in Lemma \ref{lem-image of the points in the line alpha and beta}. Put $ t=0 $ for $ h\left(-\beta - \frac{2d}{c} \right) $ and put $ t = \frac{1}{2} $ for $ h \left(-\frac{\alpha + \beta}{2} - \frac{2d}{c} \right) $ in Lemma \ref{lem-image of the points in the line alpha and beta}. 
\end{proof}
\smallskip
\begin{lem} \label{lem-image of circle 1}
Let $ g $ be the hyperbolic M\"obius map $ g(z) = \frac{az + b}{cz + d} $ where $ ad - bc = 1 $ and $ c \neq 0 $. Let $ h $ be another M\"obius map as follows
$$  h(z) = \frac{z- \beta}{z - \alpha} . $$
where $ \alpha $ and $ \beta $ be the attracting and the repelling fixed point of $ g $ respectively. Then
$$
\partial S \left( \frac{1}{\sqrt{k}} \right) = \left\{ z \colon \, \left| z + \frac{d}{c} \right| = \left| \frac{d}{c} + \beta \right| \,\right\}
$$
and 
$$
h\left( \partial S \left( \frac{1}{\sqrt{k}} \right) \right)
 = \left\{ w \colon \, \left| w - \frac{1}{k + 1} \right| = \frac{1}{k+1} \,\right\}
$$
where $ k = \frac{1}{(c\beta + d)^2} >1 $. 
\end{lem}

\begin{proof}
The fact that $ k = \frac{1}{(c\beta+ d)^2} $ implies \ 
$  \frac{| c\beta + d|}{|c|} = \frac{1}{\sqrt{k}|c|} $. Thus by Lemma \ref{lem-invariant region of Mobius}, we have 
$$
\partial S \left( \frac{1}{\sqrt{k}} \right) = \left\{ z \colon \, \left| z + \frac{d}{c} \right| = \left| \frac{d}{c} + \beta \right| \,\right\} .
$$
The extended straight line $ L_{\infty} $ contains $ \beta $ and $ -\beta - \frac{2d}{c} $ by Lemma \ref{lem-image of -d/c under h}. The midpoint of the line segment between these two points is $ -\frac{d}{c} $, which is the center of $ \partial S \left( \frac{1}{\sqrt{k}} \right) $. The half of the distance between $ \beta $ and $ -\beta - \frac{2d}{c} $ is the radius. Moreover, $ \partial S \left( \frac{1}{\sqrt{k}} \right) $ meets $ L_{\infty} $ at two points $ \beta $ and $ -\beta - \frac{2d}{c} $ at right angle because $ L_{\infty} $ goes through the center of $ \partial S \left( \frac{1}{\sqrt{k}} \right) $. Since $ h $ is conformal, $ F(L_{\infty}) $ also meets $ F \left(\partial S \left( \frac{1}{\sqrt{k}} \right) \right) $ at $ F(\beta) $ and $ F \left(-\beta - \frac{2d}{c} \right) $ at right angle. 
%
%
%
%
$ h(\beta) =0 $ by the definition of $ h $. Corollary \ref{cor-image under h at t0 and t half} implies that 
\begin{align} \label{eq-radius of circle 1}
h \left(-\beta - \dfrac{2d}{c} \right) 
= \frac{2}{k+1} .
\end{align}
Then center of the circle $ h \left( \partial S \left( \frac{1}{\sqrt{k}} \right) \right) $ is the midpoint of the line segment between $ 0 $ and $ \frac{2}{k+1} $. Hence, $ h \left(\partial S \left( \frac{1}{\sqrt{k}} \right) \right) $ is the following circle 
\begin{align*}
\left\{ w \,\colon \left| w- \frac{1}{k+1} \right| = \frac{1}{k+1} \, \right\} .
\end{align*}
\end{proof}

\begin{cor} \label{cor-equation for distance and k}
The following equations hold 
\begin{align*}
\left| \frac{c\beta + d}{c} \right| = \frac{1}{k-1}\,|\alpha - \beta | \quad \ \text{and} \quad \ \frac{1}{|c|} = \frac{\sqrt{k}}{k-1}\,|\alpha - \beta | 
\end{align*}
where $ k = \frac{1}{(c\beta+ d)^2} $. 
\end{cor}

\begin{proof}
Observe that $ h^{-1}(w) = \frac{\alpha w - \beta}{w -1} $. Thus we have that 
$$
h^{-1} \left(\dfrac{2}{k+1} \right) = - \frac{k+1}{k-1}\,(\alpha - \beta) + \alpha 
$$
and $ h^{-1}(0) = \beta $. The distance between the above two points is the diameter of $ \partial S \left( \frac{1}{\sqrt{k}} \right) $. Then the half of this distance is the radius of $ \partial S \left( \frac{1}{\sqrt{k}} \right) $ as follows  
\begin{align*}
\frac{1}{2} \left| \,- \frac{k+1}{k-1}\,(\alpha - \beta) + \alpha - \beta \, \right| = \frac{1}{k-1}\,| \alpha - \beta | .
\end{align*}
However, $ \left| \frac{c\beta + d}{c} \right| $ is also the radius of $ \partial S \left( \frac{1}{\sqrt{k}} \right) $  by Lemma \ref{lem-image of circle 1}. Hence, we have $ \left| \frac{c\beta + d}{c} \right| = \frac{1}{k-1}\,|\alpha - \beta | $ and moreover, since $ \frac{1}{\sqrt{k}} = | c\beta + d| $, the equation $ \frac{1}{|c|} = \frac{\sqrt{k}}{k-1}\,|\alpha - \beta | $ holds. 
\end{proof}
\smallskip
\begin{rk}
\noindent 
The upper bound of $ \frac{1}{|c|} $ for every $ k>1 $ is the distance between $ -\frac{d}{c} $ and $ \frac{\alpha + \beta}{2} $ because
\begin{align*}
\left| \frac{\alpha + \beta}{2} - \left( -\frac{d}{c} \right) \right| = \left| \frac{a-d}{2c} + \frac{d}{c} \right| = \frac{a+d}{2|c|} > \frac{2}{2|c|} = \frac{1}{|c|} .
\end{align*}
Recall that $ \mathrm{tr}(g) = a+d = 2 + \tau $. Thus the equation, $ \frac{a+d}{2|c|} = \frac{1 + \frac{\tau}{2}}{|c|} $ holds. Then the disk $ \hat{\C} \setminus S \left(1 + \frac{\tau}{2} \right) $ as follows 
\begin{align} \label{eq-disk containing s1}
\hat{\C} \setminus S \left(1 + \frac{\tau}{2} \right) = \left\{ z \colon  \left| z + \frac{d}{c} \right| \leq \left| \frac{\alpha + \beta}{2} + \frac{d}{c} \right| \, \right\} .
\end{align}
Moreover, this disk contains compactly the disk $ \hat{\C} \setminus S(1) $. 
\end{rk}
\smallskip

\begin{prop} \label{cor-circle for escaping boundary}
The circle $ h \left( \partial S \left(1 + \frac{\tau}{2} \right) \right) $ is as follows 
\begin{align} \label{eq-boundary of S 5/4}
\left\{ w \colon \; \left| w - \frac{1}{2} \left(\frac{k+3}{3k+1} - 1 \right) \right| = \frac{1}{2} \left( \frac{k+3}{3k+1} + 1 \right) \right\} .
\end{align}
\end{prop}

\begin{proof}
By the definition of $ \partial S(r) $ for $ r>0 $, the line connecting the fixed points $ \alpha $ and $ \beta $, say $ L $, goes through the center of the circle $ \partial S(r) $. In other words, the circle and line meet at two points at right angle. Since $ h $ is conformal, $ h(\partial S(r)) $ meets also the real line at two points at right angle.
\smallskip \\
Observe that two points in the set $ \partial S \left(1 + \frac{\tau}{2} \right) \cap L $ are $ - \frac{\alpha + \beta}{2} - \frac{2d}{c} $ and $ \frac{\alpha + \beta}{2} $. Thus it suffice to show that the image of two points in $ \partial S \left(1 + \frac{\tau}{2} \right) \cap L $ under $ h $ is the points $ \frac{k+3}{3k+1} $ and $ -1 $. 
%
%
In Lemma \ref{lem-image of the points in the line alpha and beta}, put $ t = \frac{1}{2} $. Hence, we have the following equations 
$$
h \left( \frac{\alpha + \beta}{2} - \dfrac{2d}{c} \right) = \frac{k+3}{3k+1}  \quad \text{and} \quad h \left( \frac{\alpha + \beta}{2} \right) = -1 
$$
for each $ k > 1 $. Then the midpoint of the two points $ \frac{k+3}{3k+1} $ and $ -1 $ is the center of the circle $ h \left( \partial S \left(1 + \frac{\tau}{2} \right) \right) $. Moreover, the half of the distance between these two points is the radius of $ h \left( \partial S \left(1 + \frac{\tau}{2} \right) \right) $. 
\end{proof}
\medskip
%
%
\begin{figure}
    \centering
    \includegraphics[scale=0.8]{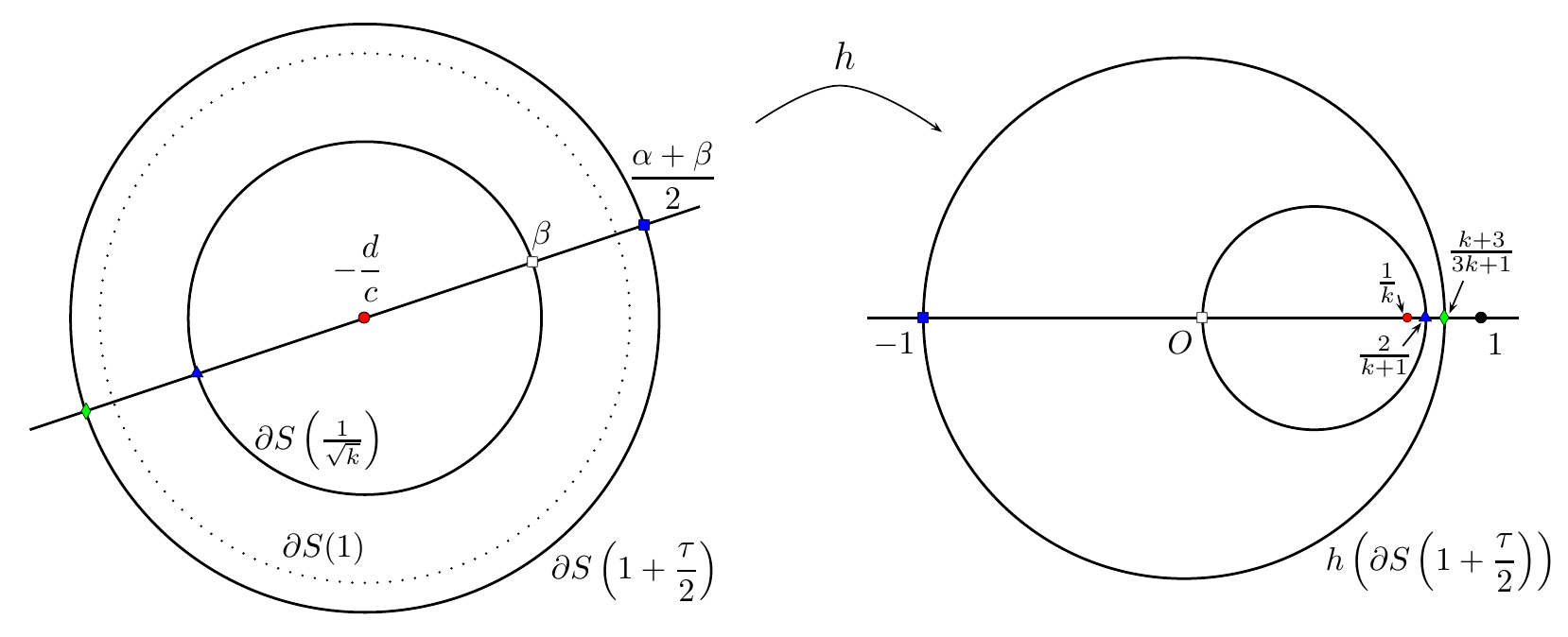}   
    \caption{Concentric circles and its images under $h$}
\end{figure}
%
\begin{rk}
Observe that the map $ y(k) = \frac{k+3}{3k+1} $ is a decreasing function for $ k > 0 $ and $ \frac{1}{3} < \frac{k+3}{3k+1} <1 $ for $ k > 1 $. $ y(1) = 1 $ and $ \lim_{k \rightarrow \infty} \frac{k+3}{3k+1} = \frac{1}{3} $. Then the circle in Corollary \ref{cor-circle for escaping boundary}, $ h \left( \partial S \left(1 + \frac{\tau}{2} \right) \right) $ is contained in the unit disk $ \left\{ w \colon |w| \leq 1 \right\} $ for all $ 1 < k < \infty $. 
\end{rk}

\medskip

\section{Avoided region} \label{sec-Avoided region}
The map $ g $ is the hyperbolic M\"obius map as follows 
$$ g(z) = \frac{az + b}{cz + d} $$
for $ ad - bc =1 $ and $ c \neq 0 $. Since the point $ \infty $ is not a fixed point of $ g $, the preimage of $ \infty $ under $ g $, namely, $ g^{-1}(\infty) $ is in the complex plane. For a given $ \e > 0 $, the sequence $ \{ a_i \}_{i \in \N_0 } $ satisfying that  
$$ | a_{i+1} - g(a_i)| \leq \e $$ 
contains $ g^{-1}(\infty) $,  say $ a_k $, then $ | a_{k+1} - \infty | $ is not bounded where $ | \cdot | $ is the absolute value of the complex number. 
In order to exclude $ g^{-1}(\infty) $ in the whole sequence $ \{ a_i \}_{i \in \N_0 } $, the region $ \RR_g $ is defined such that if the initial point of the sequence, $ a_0 $ is not in $ \RR_g $, then the whole sequence $ \{ a_i \}_{i \in \N_0 } $ cannot be in the same region $ \RR_g $. Let the forward orbit of $ p $ under $ F $ be the set $ \{ F(p),F^2(p), \ldots ,F^n(p) , \ldots \} $ and denote it by $ \Orb_{\N}(p, F) $. 
\bigskip 
\begin{definition} \label{def-avoided region}
Let $ F $ be the map on $ \hat{\C} $ which does not fix $ \infty $. Avoided region for the sequence $ \{ a_i \}_{i \in \N_0 } $ satisfying $ | a_{i+1} - F(a_i)| \leq \e $ for a given $ \e >0 $ which is denoted by $ \RR_F \subset \C $ is defined as follows
\begin{enumerate}
\item $ \hat{\C} \setminus \RR_F $ is (forward) invariant under $ F $, that is, $ F(\hat{\C} \setminus \RR_F) \subset \hat{\C} \setminus \RR_F $.
\item For any given initial point $ a_0 $ in $ \C \setminus \RR_F $, all points in the sequence $ \{ a_i \}_{i \in \N_0 } $ satisfying $ | a_{i+1} - F(a_i)| \leq \e $ are in $ \C \setminus \RR_F $. 
\end{enumerate}
If $ \RR_F $ contains $ \Orb_{\N}(p, F^{-1}) $ where $ p \in \hat{\C} $, then it is called the avoided region at $ p $ and is denoted by $ \RR_F(p) $. 
\end{definition}
\smallskip 
\noindent In the above definition, the avoided region does not have to be connected.  
\begin{rk}
The set $ \hat{\C} \setminus S(1+t) $ in Proposition \ref{prop-stability on S1+t} is an avoided region at $ \infty $. However, avoided region $ \hat{\C} \setminus S \left(1+\frac{\tau}{2} \right) $ can be extended to some neighborhood of $ \Orb_{\N}(\infty, g^{-1}) $, which is denoted to be $ \RR_{g}(\infty) $ in the following proposition. 
\end{rk}
%
%
%
%

%
%
\begin{figure}
    \centering
    \begin{subfigure}[b]{0.48\textwidth}
        \includegraphics[width=\textwidth]{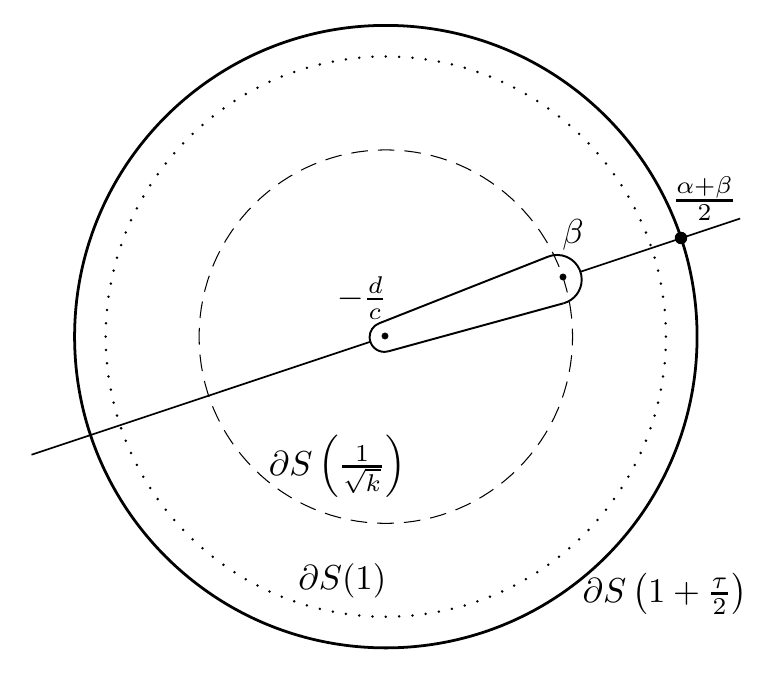}
        \caption{Avoided region $\RR_{g}(\infty) $}
        \label{fig:avoided region for g}
    \end{subfigure}
    ~ 
    \begin{subfigure}[b]{0.48\textwidth}
        \includegraphics[width=\textwidth]{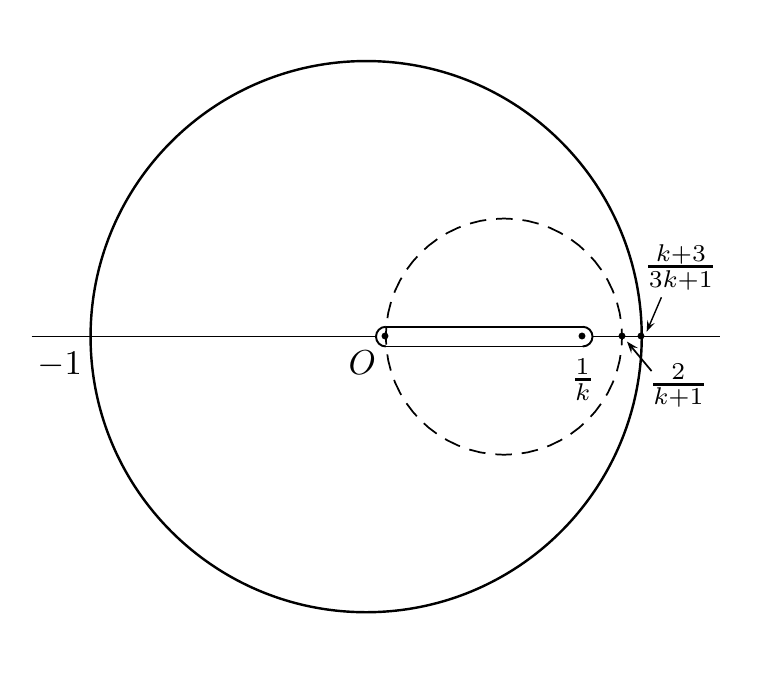}
        \caption{Avoided region $\RR_{f}(1)$}
        \label{fig:avoided region for f}
    \end{subfigure}
    \caption{Avoided regions for hyerbolic M\"obius maps}\label{fig:Avoided regions}
\end{figure}
Recall that for a complex number $ z $, $ \re\,z $ is the real part of $ z $ and $ \im\,z $ is the imaginary part of $ z $. The complex conjugate of $ z $ is denoted by $ \bar z $. 
%
\begin{prop} \label{prop-avoided region for f}
Define the following regions  
\begin{align*}
\RR_{1}^t &= \left\{ w \colon |w| \leq \frac{t\delta}{k-1} \right\} \cap \left\{ w \colon \re\,w \leq 0 \right\} \\[0.3em]
\RR_2^t &=  \left\{ w \colon -\frac{t\delta}{k-1} \leq \im \,w \leq \frac{t\delta}{k-1} \right\} \cap \left\{ w \colon 0 \leq \re\,w \leq \frac{1}{k} \right\}  \\[0.3em]
\RR_3^t &= \left\{ w \colon \left| w - \frac{1}{k} \right| \leq \frac{t\delta}{k-1} \right\} \cap \left\{ w \colon \re\,w \geq \frac{1}{k} \right\} 
\end{align*}
for $ t \geq 1 $ and denote the union $ \RR^t_1 \cup \RR^t_2 \cup \RR^t_3 $ by $ \RR^t $. 
Let the map $ f $ be the dilation, $ f(w) = kw $ for the given number $ k>1 $. Let $ \{ c_i \}_{i \in \N_0} $ be the sequence for a given $ \delta > 0 $ satisfying that
\begin{align} \label{eq-sequence c-n for f}
| c_{i+1} - f(c_i) | < \delta 
\end{align}
for all $ i \in \N_0 $. If $ \delta < \left(1- \frac{1}{k} \right)^2 $, then the set $ \RR^t $ is an avoided region for $ \{ c_i \}_{i \in \N_0} $, namely, $ \RR_f $. Furthermore, $ \RR_f $ contains the the orbit, $ \Orb_{\N}(1, f^{-1}) $. Hence, the avoided region can be chosen as $ \RR_f(1) $. 
\end{prop}
The proof of Proposition \ref{prop-avoided region for f} requires the combined result of lemmas as follows. 
\smallskip
\begin{lem} \label{lem-1st lemma for avoided region}
Let $ f $ be the map $ f(w) = kw $ for $ k > 1 $. The sequence $ \{ c_i \}_{i \in \N_0} $ is defined in \eqref{eq-sequence c-n for f} for $ \delta > 0 $. Consider the following regions
\begin{align*}
\DD_1^t = \left\{ w \colon |w| \leq \frac{t\delta}{k-1} \right\} , \quad 
\DD_2^t = \left\{ w \colon -\frac{t\delta}{k-1} \leq \im \,w \leq \frac{t\delta}{k-1} \right\} 
\end{align*}
for $ \delta > 0 $ and $ t \geq 1 $. If $ c_0 \in \hat{\C} \setminus \DD_j^t $ for $ j =1,2 $, then the whole sequence $ \{ c_i \}_{i \in \N_0} $ is contained in the same set $ \hat{\C} \setminus \DD_j^t $ respectively. 
\end{lem}

\begin{proof}
Suppose firstly that $ |c_0| > \frac{t\delta}{k-1} $. Thus $ |f(c_0)| > \frac{kt\delta}{k-1} $. The inequality $ | c_{i+1} - f(c_i) | < \delta $ implies that 
\begin{equation} 
\delta > | c_1 - f(c_0) | \geq |\, |c_1| - |f(c_0)| \,| 
\end{equation}
\begin{equation} \label{eq-triangular ineq 2}
|c_1| > |f(c_0)| - \delta > \dfrac{kt\delta}{k-1} - t\delta = \dfrac{t\delta}{k-1} 
\end{equation}
for $ t \geq 1 $. Then $ c_1 \in \hat{\C} \setminus \DD_1^t $. By induction the whole sequence is also contained in $ \hat{\C} \setminus \DD_1^t $. 
\smallskip \\
Similarly, suppose that $ c_0 \in \DD_2^t $. Since $ | \im (c_1 - f(c_0)) | \leq | c_1 - f(c_0) | < \delta $, the assumption $ c_0 \in \DD_2^t $ implies the following inequality.
\begin{align*}
| \im\,c_1 | &> | \im\,f(c_0) | - \delta = k \,| \im\,c_0 | - \delta > \dfrac{kt\delta}{k-1} - t\delta = \dfrac{t\delta}{k-1} 
\end{align*}
for $ t \geq 1 $. Then $ c_1 \in \DD_2^t $. Hence, by induction the whole sequence is also contained in $ \hat{\C} \setminus \DD_2^t $. 
\end{proof}
Observe that $ \DD_1^t \subset \DD_2^t $ for every $ t \geq 1 $ in Lemma \ref{lem-1st lemma for avoided region}. The region $ \RR_1^t $ is the half disk of $ \DD_1^t $ in the left half complex plane. Thus if a point $ c_j $ in $ \hat{\C} \setminus \RR_1^t $ and $ c_{j+1} $ is also contained in the set $ \{ w \colon \re\,w \leq 0 \} $ for some $ j \in \N_0 $, then $ c_{j+1} $ is contained in $ \big( \hat{\C} \setminus \RR_1^t \big) \cap \{ w \colon \re\,w \leq 0 \} $. However, $ c_{j+1} $ may be in the right half complex plane. Thus in order to construct the avoided region in Proposition \ref{prop-avoided region for f}, another lemma is required as follows.
\begin{figure}
    \centering
    \includegraphics[scale=0.85]{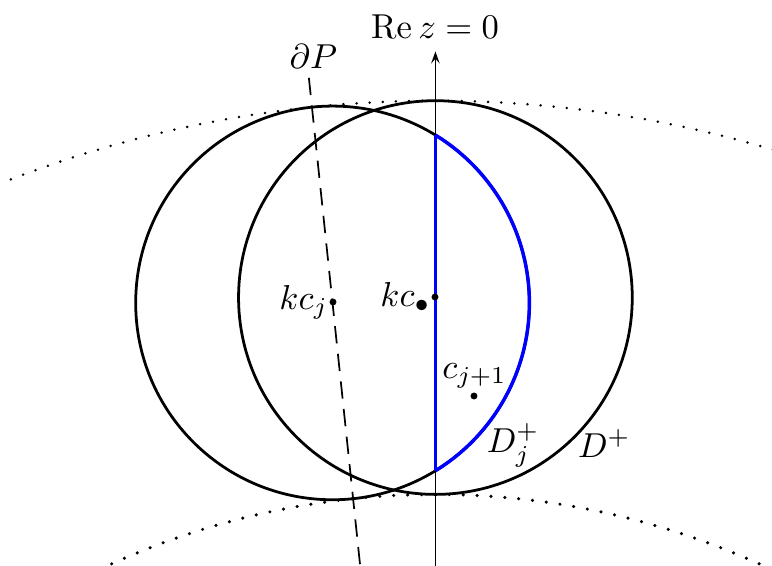}   
    \caption{}
\end{figure}
\begin{lem} \label{lem-2nd lemma for avoided region}
Let $ f $ be the map $ f(w) = kw $ for $ k > 1 $. The sequence $ \{ c_i \}_{i \in \N_0} $ is defined in \eqref{eq-sequence c-n for f} for $ \delta > 0 $. Suppose that $ c_j $ for some $ j \in \N_0 $ is contained in the set $ \big( \hat{\C} \setminus \DD_1^t \big) \cap \{ w \colon \re\,w \leq 0 \} $ but $ \re\,c_{j+1} \geq 0 $. 
Then $ c_{j+1} $ satisfies that $ | \im\,c_{j+1} | > \frac{t\delta}{k-1} $, that is, $ c_{j+1} \in \hat{\C} \setminus \DD_2^t $ for $ t \geq 1 $. 
\end{lem}

\begin{proof}
By the assumption we have 
\begin{align*}
c_j \in \left\{ w \colon |w| > \frac{t\delta}{k-1} \right\} \cap \{ w \colon \re\,w \leq 0 \}
\end{align*}
for some given $ j \in \N_0 $. Let $ c_{\bullet} $ be the purely imaginary number satisfying $ |c_j| = |c_{\bullet}| $ and $ \im\,c_j $ and $ \im\,c_{\bullet} $ has the same sign.   Without loss of generality, we may assume that $ \im\,c_j > 0 $ and $ \im\,c_{\bullet} > 0 $. The proof of other case is similar. 
Define the sets as follows 
\begin{align*}
D^+ &= \{ w \colon  | w -kc_{\bullet} | < \delta \} \cap \{ w \colon \re\,w \geq 0 \} \\ 
D^+_j &= \{ w \colon  | w -kc_{j} | < \delta \} \cap \{ w \colon \re\,w \geq 0 \} .
\end{align*}
Consider the region
\begin{align*}
P = \{ z \colon | z- kc_{\bullet} | < | z- kc_j| \}
\end{align*}
We show that $ D^+_j \subseteq D^+ $ where $ \re\,c_j \leq 0 $ using the following claim. 
\smallskip \\
\textit{claim}: For any $ z \in \big( \hat{\C} \setminus P \big) \cap \{ z \colon \im\,z \geq 0 \} $, the real part of $ z $ is negative, namely, $ \re\,z \leq 0 $. The number $ z $ satisfies the equivalent inequalities 
\begin{align*}
& \quad | z- kc_{\bullet} | \geq | z- kc_j| \\[0.2em]
\Longleftrightarrow & \quad | z- kc_{\bullet} |^2 \geq | z- kc_j|^2 \\[0.2em]
\Longleftrightarrow &  - kc_{\bullet} \bar{z} - k\bar{c_{\bullet}} z \geq - kc_j \bar{z} - k\bar{c_j} z \\[0.2em]
\Longleftrightarrow & -\re\,(\bar{c_{\bullet}} z) \geq -\re\,(\bar{c_j} z) \\[0.2em]
\Longleftrightarrow & \quad \re\,((\overline{c_j - c_{\bullet}}) z) \geq 0
\end{align*}
Recall that $ \re\,c_j \leq 0 $, $ \re\,c_{\bullet} =0 $ and $ \im\,c_{\bullet} \geq \im\,c_j > 0 $ because $ |c_j| = |c_{\bullet}| $. Thus $ \overline{c_j - c_{\bullet}} = -a + bi $ for some $ a,\ b > 0 $. Denote $ z = x +yi $ for $ y > 0 $. Then $ \re\,((\overline{c_j - c_{\bullet}}) z) = \re\,\big((-a+bi)(x+yi)\big) = -ax - by \geq 0 $. Then $ x \leq 0 $, that is, $ \re\,z \leq 0 $ is the necessary condition for the inequality. The proof of the claim is complete. \smallskip \\
If $ \re\,w \geq 0 $ and $ \im\,w \geq 0 $, then $ w \in P $. Thus  
\begin{align*}
| w- kc_{\bullet} | < | w- kc_j| < \delta .
\end{align*}
Then $ D^+_j \subseteq D^+ $. For every $ w \in D^+ $, the following inequality holds
\begin{align*}
& | \im\,w - \im\,kc_{\bullet}| \leq | w- kc_{\bullet} | < \delta \\
\Longrightarrow \  | \im\,w | & > k\,|\im\,c_{\bullet}| - \delta \\
& = k | c_{\bullet} | - \delta = k | c_j | - \delta  \\
& > \frac{tk\delta}{k-1} - t\delta \\
& = \frac{t\delta}{k-1}
\end{align*}  
Then $ D^+ $ is disjoint from $ \DD_2^t $. The fact that $ c_{j+1} \in D^+_j $ implies that $ c_{j+1} \notin \DD_2^t $ for $ t \geq 1 $. Hence, $ c_{j+1} \in \hat{\C} \setminus \DD_2^t $ for $ t \geq 1 $. 
\end{proof}
\smallskip 

\begin{lem} \label{lem-3rd lemma for avoided region}
Let $ f $ be the map $ f(w) = kw $ for $ k > 1 $. The sequence $ \{ c_i \}_{i \in \N_0} $ is defined in \eqref{eq-sequence c-n for f} for $ \delta > 0 $. Suppose that $ c_j \in \big(\hat{\C} \setminus \RR_3^t \big) \cap \left\{ w \colon \re\,w \geq \frac{1}{k} \right\} $, that is, 
\begin{align} \label{eq-complement of half disk}
c_j \in \left\{ w \colon \left| w - \frac{1}{k} \right| > \frac{t\delta}{k-1} \right\} \cap \left\{ w \colon \re\,w \geq \frac{1}{k} \right\} 
\end{align}
for $ t \geq 1 $ and for some $ j \in \N_0 $. If $ \delta < \left( \frac{k-1}{k} \right)^2 $, then $ c_{j+1} $ is contained in 
$ \hat{\C} \setminus \RR_3^t \cap \left\{ w \colon \re\,w \geq \frac{1}{k} \right\} $ for $ t \geq 1 $. 
\end{lem}

\begin{proof}
Observe that $ f(\hat{\C} \setminus \RR_3^t) $ is the half disk. Since $ | c_{j+1} - f(c_j) | < \delta $, we have that 
$ |c_{j+1}| > |f(c_j)| - \delta \geq 1- \delta $. Moreover, $ |c_j| < \frac{1}{k} + \frac{t\delta}{k-1} $. Thus it suffice to show that $ 1 - \delta > \frac{1}{k} + \frac{t\delta}{k-1} $. Then
\begin{align*}
1 - \delta > \frac{1}{k} + \frac{t\delta}{k-1} \Longleftrightarrow \ 1 - \frac{1}{k} & > \frac{t\delta}{k-1} + \delta = \left( \frac{t}{k-1} + 1 \right) \delta \\
 & \geq \left( \frac{k+t-1}{k-1} \right) \delta .
\end{align*}
Hence, $ \delta < \frac{(k-1)^2}{k(k+t-1)} \leq \left( \frac{k-1}{k} \right)^2 $ for all $ t \geq 1 $. 
\end{proof}

\smallskip
\begin{proof}[proof of Proposition \ref{prop-avoided region for f}]
The definition of $ \RR^t $ and $ \DD_2^t $ implies that $ \RR^t \subset \DD_2^t $. If $ c_j \in \DD_2^t $, then $ c_{j+1} \in \DD_2^t $ by Lemma \ref{lem-1st lemma for avoided region}. Thus we may assume that $ \mathrm{Re}\,c_j \leq 0 $ or $ \mathrm{Re}\,c_j \geq \frac{1}{k} $. If $ c_j \in \hat{\C} \setminus R_1^t \cap \{w \colon \mathrm{Re} \leq 0 \} $, then $ c_{j+1} $ is contained in the same set by Lemma \ref{lem-2nd lemma for avoided region}. Similarly, if $ c_j \in \hat{\C} \setminus R_3^t \cap \{w \colon \mathrm{Re} \leq 0 \} $, then $ c_{j+1} $ is contained in the same set by Lemma \ref{lem-3rd lemma for avoided region}. Hence, $ \RR^t = \RR^t_1 \cup \RR^t_2 \cup \RR^t_3 $ can be an avoided region $ \RR_f $ for $ t \geq 1 $. Moreover, $ \Orb_{\N}(1, f^{-1}) $ is contained in the line segment connecting $ 0 $ and $ \frac{1}{k} $. 
\end{proof}

We define the avoided region $ \RR_f(1) $ as $ \RR_1^k \cup \RR^1_2 \cup \RR^1_3 $ and $ \RR_g(\infty) $ as $ h^{-1}(\RR_f(1)) $. 

\section{Escaping time from the region}

Let the sequence $ \{ c_i \}_{i \in \N_0 } $ satisfies the following
\begin{align} \label{eq-sequence under f}
| c_{i+1} - f(c_i) | \leq \delta
\end{align}
for all $ i \in \N_0 $. For the given region $ R $, assume that $ c_0 \in R $. If the distance between $ c_n $ and the closure of $ R $ is positive for all $ n \geq N $ for some $ N \in \N $, then $ N $ is called {\em escaping time} of the sequence $ \{ c_i \}_{i \in \N_0 } $ from $ R $ under $ f $. If the escaping time $ N $ is independent of the initial point $ c_0 $ in $ R $, then the number $ N $ is called {\em uniformly} escaping time. Denote the ball of which center is the origin with radius $ r>0 $ by $ B(0,r) $. 
\begin{lem} \label{lem-escaping time under f}
Let $ \{ c_i \}_{i \in \N_0 } $ be the sequence defined in the equation \eqref{eq-sequence under f} where $ f(w) = kw $ for $ k>1 $. Suppose that $ c_0 \in E_f $ where $ E_f $ is defined as the region $ h\Big( \hat{\C} \setminus S \left( 1+\frac{\tau}{2} \right) \Big) \setminus B \left(0,\, \frac{k \delta}{k-1} \right) $. Then the (uniformly) escaping time $ N $ from the region $ E_f $ under $ f $ satisfies the following inequality
\begin{align*}
N > \log \left( \frac{1}{\delta} \left(\frac{k+3}{3k+1} + 1 \right)  + 1 \right) \bigg/ \log k 
\end{align*}
for small enough $ \delta > 0 $. 
\end{lem}

\begin{proof}
By triangular inequality, we have 
\begin{align*}
| f^n(c_0) - c_0 | &\leq  | f^n(c_0) - f^{n-1}(c_1) | + | f^{n-1}(c_1) - f^{n-2}(c_2) | + \\
 & \qquad \cdots + | f^2(c_{n-2}) - f(c_{n-1})| +  | f(c_{n-1}) - c_{n}| + |c_n - c_0 | \\
&= \sum^{n}_{j=1} k^{n-j}|f(c_{j-1}) - c_j| + |c_n - c_0 | \\
&\leq \frac{k^n-1}{k-1}\, \delta + |c_n - c_0 | .
\end{align*}
Thus we have 
\begin{align*}
|c_n - c_0 | &\geq | f^n(c_0) - c_0 | -\frac{k^n-1}{k-1}\, \delta \\
&= (k^n-1)|c_0| - \frac{k^n-1}{k-1}\,\delta \\
&= (k^n-1) \left(|c_0| -\frac{\delta}{k-1} \right) \\
&\geq (k^n-1) \left(\frac{k\delta}{k-1} -\frac{\delta}{k-1} \right) \\
&= (k^n-1)\, \delta
\end{align*}
Hence, the escaping time $ N $ satisfies 
if the inequality $ (k^N-1)\, \delta > \frac{k+3}{3k+1} + 1 $ holds, then $ |c_n - c_0 | $ is greater than the diameter of $ h\Big(S \left( 1+\frac{\tau}{2} \right) \Big) $ for all $ n \geq N $ by \eqref{eq-boundary of S 5/4} in Proposition \ref{cor-circle for escaping boundary}. Hence, $ N $ is the uniformly escaping time where $ N > \frac{\log \left( \frac{1}{\delta} \left(\frac{k+3}{3k+1} + 1 \right)  + 1 \right)}{\log k} $. 
\end{proof}

\medskip

\begin{rk}
The inequality $ \frac{1}{3} < \frac{k+3}{3k+1} < 1 $ for $ k > 1 $ implies that a upper bound of the uniformly escaping time is $ N_0 > \log \left( \frac{2}{\delta} + 1 \right) \big/ \log k  $. 
\end{rk}

\medskip

\noindent The sequence $ \{ a_i \}_{i \in \N_0} $ is defined as the set each of which element $ a_i = h(c_i) $ for every $ i \in \N_0 $ where $ \{ c_i \}_{i \in \N_0} $ is defined in \eqref{eq-sequence under f}. Recall that $ f $ is the map $ h \circ g \circ h^{-1} $. Denote the radius of the ball $ h^{-1} \big( B(c_j, \delta) \big) $ by $ \e_j $ for $ j \in \N_0 $. Then the sequence $ \{ a_i \}_{i \in \N_0} $ as follows  
\begin{align}  \label{eq-sequence a-i without epsilon}
| a_{i+1} - g(a_i) | \leq \e_i .
\end{align}
corresponds $ \{ c_i \}_{i \in \N_0} $ by the conjugation $ h $. Then the escaping time of $ \{ a_i \}_{i \in \N_0} $ from $ h^{-1}(E_f) $ under $ g $ is the same as that of $ \{ c_i \}_{i \in \N_0} $ from $ E_f $ under $ f $ in Lemma \ref{lem-escaping time under f}. Furthermore, since $ h $ is uniformly continuous on the closure of $ \hat{\C} \setminus S\left(1 + \frac{\tau}{2} \right) $ under Euclidean metric, there exists $ \e > 0 $ such that $ h \big( B(a_j,\e) \big) \subset B(c_j,\delta) $ for $ j =1,2, \ldots N_1 $ for all $ c_j \in E_f $. Thus we obtain the following Proposition.
\begin{prop} \label{prop-same escaping time under g}
Let $ \{ c_i \}_{i \in \N_0} $ be the sequence satisfying 
\begin{align*}  
| c_{i+1} - f(c_i) | \leq \delta 
\end{align*}
where $ f(w) = kw $ for $ k>1 $ on $ E_f $ defined in Lemma \ref{lem-escaping time under f}. Let $ N $ be the (uniformly) escaping time from $ E_f $ Let $ \{ a_i \}_{i \in \N_0} $ be the sequence satisfying $ a_i = h(c_i) $ for every $ i \in \N_0 $. Then there exists $ \e > 0 $ such that if $ \{ a_i \}_{i \in \N_0} $ satisfies that 
\begin{align}  \label{eq-sequence a-i without epsilon 2}
| a_{i+1} - g(a_i) | \leq \e 
\end{align}
on $ h^{-1}(E_f) $ for $ i = 1,2,\ldots, N-1 $, then the escaping time from $ h^{-1}(E_f) $ under $ g $ is also $ N $.
\end{prop}

\begin{rk}
The definition of $ E_f $ implies that the set $ h^{-1}(E_f) $ is the set, $ \big( \hat{\C} \setminus S \left( 1+\frac{\tau}{2} \right) \big) \setminus h^{-1}\left( B \left(0,\, \frac{k \delta}{k-1} \right)\right) $. 
\end{rk}

\medskip

\section{Hyers-Ulam stability on the complement of the  avoided region}
Hyers-Ulam stability of hyperbolic M\"obius map requires two different regions, one of which is $ S(1+\frac{\tau}{2}) $ where $ \mathrm{tr}(g) = 2 + \tau $ for $ \tau > 0 $. The other region is the $ \left(\hat{\C} \setminus S \left(1+\frac{\tau}{2} \right) \right) \setminus \RR_{g}(\infty) $ where $ \RR_{g}(\infty) = h^{-1} \left( \RR_{f}(1) \right) $ is the avoided region defined in Section \ref{sec-Avoided region}. We show that Hyers-Ulam stability on the region $ \left(\hat{\C} \setminus S \left(1+\frac{\tau}{2} \right) \right) \setminus \RR_{g}(\infty) $ for finite time bounded by the uniformly escaping time. Then this stability and Proposition \ref{prop-stability on S1+t} implies Hyers-Ulam stability of $ g $ on the set $ \hat{\C} \setminus \RR_{g}(\infty) $. 
%
%
%
\medskip
%
%
%
%
%

%
%

\begin{lem} \label{lem-disk with radius epsilon}
The avoided region $ h^{-1} \left( \RR_{f}(1) \right) $ contains the disk 
\begin{align*}
D_{\tilde{\e}} = \left\{ z \colon \left| z + \frac{d}{c} \right| < \tilde{\e} \right\}
\end{align*}
where $ \tilde{\e} = \frac{k^2 \delta}{(k-1)^3} \,| \alpha - \beta | $. 
\end{lem}

\begin{proof}
The avoided region $ \RR_{f}(1) $ contains the disk
\begin{align*}
B \left(\frac{1}{k},\; \frac{\delta}{k-1} \right) = \left\{ w \colon \left| w- \frac{1}{k} \right| \leq \frac{\delta}{k-1} \right\} .
\end{align*}
Since $ h^{-1}\left(\frac{1}{k} \right) = -\frac{d}{c} $, the avoided region $ h^{-1} \left( \RR_{f}(1) \right) $ contains a small disk of which center is $ -\frac{d}{c} $. Thus the number $ \tilde{\e} $ is either the radius of the circle, $ h^{-1}\left(B \left(\frac{1}{k},\; \frac{\delta}{k-1} \right) \right) $ or the distance between $ -\frac{d}{c} $ and $ h^{-1}\left(\frac{1}{k} + \frac{\delta}{k-1} \right) $. Denote the point $ \frac{1}{k} + \frac{\delta}{k-1} $ by $ \frac{t_0}{k} $, that is, choose $ t_0 = 1 + \frac{k \delta}{k-1} $. Thus we have
\begin{align} \label{eq-image under h-1 for extreme end}
h^{-1}\left( \frac{t_0}{k}\right) = \frac{k \beta - t_0 \alpha}{k - t_0} = \left( 1 + \frac{t_0}{k - t_0} \right) \beta - \frac{t_0}{k - t_0}\, \alpha .
\end{align}
By Lemma \ref{lem-image of -d/c under h}, we have 
\begin{align*} 
-\frac{d}{c} = \left( 1 + \frac{1}{k-1} \right) \beta - \frac{1}{k-1}\, \alpha .
\end{align*}
Then the distance between $ -\frac{d}{c} $ and $ h^{-1}\left( \frac{t_0}{k}\right) $ is as follows 
\begin{align} \label{eq-first candidate for epsilon}
& \left| \left( 1 + \frac{t_0}{k - t_0} \right) \beta - \frac{t_0}{k - t_0}\, \alpha - \left( 1 + \frac{1}{k-1} \right) \beta - \frac{1}{k-1}\, \alpha \right| \nonumber \\[0.5em]
= & \left| \frac{1}{k-1} - \frac{t_0}{k - t_0} \right| | \alpha - \beta | \nonumber \\[0.5em]
= & \left| \frac{1}{k-1} - \frac{1 + \frac{k \delta}{k-1}}{k - 1 - \frac{k \delta}{k-1}} \right| | \alpha - \beta | \nonumber \\[0.5em]
= & \left| \frac{1}{k-1} - \frac{k-1 + k\delta}{(k - 1)^2 - k\delta} \right| | \alpha - \beta | \nonumber \\[0.5em]
= & \left| \frac{-k^2 \delta}{(k-1) \{ (k-1)^2 - k\delta \} } \right| | \alpha - \beta | \nonumber \\[0.5em]
= &  \frac{k^2 \delta}{(k-1) \{ (k-1)^2 - k\delta \} } | \alpha - \beta |
\end{align}
Another candidate for $ \tilde{\e} $ is the half of diameter of $ h^{-1}\left(\RR_2 \right) $. Thus take two points, $ \frac{t_0}{k} $ and $ \frac{t_1}{k} $ in $ \RR_2 $ where $ \frac{t_1}{k} = \frac{1}{k} - \frac{\delta}{k-1} $, that is, $ t_1 = 1 - \frac{k \delta}{k-1} $. Since $ \frac{t_0}{k} - \frac{t_1}{k} $ is the diameter of the circle $ \RR_2 $, the half of the distance between $ h^{-1}\left( \frac{t_0}{k} \right) $ and $ h^{-1}\left(\frac{t_1}{k} \right) $ is the radius of $ h^{-1}\left(\RR_2 \right) $. Then the calculation in \eqref{eq-image under h-1 for extreme end} implies that  
\begin{align} \label{eq-second candidate for epsilon}
 &\frac{1}{2} \left| h^{-1}\left( \frac{t_0}{k}\right) - h^{-1}\left( \frac{t_1}{k}\right) \right| \nonumber \\[0.5em]
= &\frac{1}{2} \left| \left( 1 + \frac{t_0}{k - t_0} \right) \beta - \frac{t_0}{k - t_0}\, \alpha - \left( 1 + \frac{t_1}{k - t_1} \right) \beta - \frac{t_1}{k - t_1}\, \alpha \right| \nonumber \\[0.5em]
= &\frac{1}{2} \left| \frac{t_0}{k - t_0} - \frac{t_1}{k - t_1} \right| | \alpha - \beta | \nonumber \\[0.5em]
= &\frac{1}{2} \left| \frac{k-1 + k\delta}{(k - 1)^2 - k\delta} - \frac{k-1 - k\delta}{(k - 1)^2 + k\delta} \right| | \alpha - \beta | \nonumber \\[0.5em]
= &\frac{1}{2} \left| \frac{2k^2(k-1) \delta}{ \{(k-1)^2-k\delta \} \{(k-1)^2 + k\delta \} } \right| | \alpha - \beta | \nonumber \\[0.5em]
= & \frac{k-1}{(k-1)^2 + k\delta}\cdot \frac{k^2 \delta}{(k-1)^2 - k\delta } \, | \alpha - \beta |
\end{align}
The number $ \tilde{\e} $ have to be smaller than the both number of equation \eqref{eq-first candidate for epsilon} and \eqref{eq-second candidate for epsilon}. Hence, $ \tilde{\e} $ can be chosen as follows 
\begin{align*}
\tilde{\e} &= \dfrac{k^2 \delta}{(k-1)^3} \,| \alpha - \beta | \\
&< \frac{k-1}{(k-1)^2 + k\delta}\cdot \frac{k^2 \delta}{(k-1)^2 - k\delta } \, | \alpha - \beta | < \frac{k^2 \delta}{(k-1) \{ (k-1)^2 - k\delta \} } | \alpha - \beta | .
\end{align*}
\end{proof}

%
\begin{cor} \label{cor-upper bound of absolute value of g'}
For every $ z \in \C \setminus \RR_{g}(\infty) $, the following inequality holds
\begin{align*}
|g'(z)| \leq \frac{(k-1)^4}{k^3 \delta^2} .
\end{align*}
\end{cor}

\begin{proof}
It suffice to show the upper bound of $ |g'| $ on the region $ \C \setminus h^{-1}(\RR_2) $ because $ \RR_{g}(\infty) $ contains $ h^{-1}(\RR_2) $ and moreover, contains $ D_{\e} $ in Lemma \ref{lem-disk with radius epsilon}. Recall that 
\begin{align*}
g'(z) = \frac{1}{(cz + d)^2}
\end{align*}
Thus in the region $ \C \setminus D_{\e} $, the inequality $ |cz + d| \geq |c| \e $ holds. Then the upper bound of $ |g'| $ is as follows
\begin{align*}
|g'(z)| = \frac{1}{|cz + d|^2} \leq \frac{1}{|c|^2 \e^2} = \frac{(k-1)^6}{|c|^2k^4 \delta^2 | \alpha - \beta |^2} 
\end{align*}
by Lemma \ref{lem-disk with radius epsilon}. Moreover, Corollary \ref{cor-equation for distance and k} implies the equation
\begin{align*}
\frac{1}{|c|} = \frac{\sqrt{k}}{k-1} | \alpha - \beta | .
\end{align*}
Hence,
\begin{align*}
\frac{(k-1)^6}{|c|^2k^4 \delta^2 | \alpha - \beta |^2} = \frac{(k-1)^6}{k^4 \delta^2} \cdot \frac{k}{(k-1)^2} = \frac{(k-1)^4}{k^3 \delta^2} .
\end{align*}
\end{proof}

%
%
%
\noindent The following is the mean value inequality for holomorphic function.

\begin{lem} \label{lem-mean value thoerem for holo functin}
Let $ g $ be the holomorphic function on the convex open set $ B $ in $ \C $. Suppose that $ \displaystyle \sup_{z \in B} |g'| < \infty $. Then for any two different points $ u $ and $ v $ in $ B $, we have
\begin{align*}
\left| \frac{g(u) - g(v)}{u-v} \right| \leq 2 \sup_{z \in B} |g'| .
\end{align*}
\end{lem}

\begin{proof}
The complex mean value theorem implies that
\begin{align*}
\re(g'(p)) = \re \left( \frac{g(u) - g(v)}{u-v} \right) \quad \text{and} \quad \im(g'(q)) = \im \left( \frac{g(u) - g(v)}{u-v} \right) 
\end{align*}
where $ p $ and $ q $ are in the line segment between $ u $ and $ v $. Hence, the inequality 
$$ | \re(g'(p)) + i\,\im(g'(q)) | \leq | \re(g'(p))| + | \im(g'(q))| \leq 2 \sup_{z \in B} |g'| $$
completes the proof. 
\end{proof}

\medskip

\begin{prop} \label{prop-stability on C minus S}
Let $ g $ be the hyperbolic M\"obius map. 
For a given $ \e > 0 $, let a complex valued sequence $ \{ a_i \}_{i \in \N_0 } $ satisfies the inequality
$$
| a_{i+1} - g(a_i) | \leq \e
$$
for all $ i \in \N_0 $
. Suppose that $ a_0 \in \left(\hat{\C} \setminus S \left(1 + \frac{\tau}{2} \right) \right) \setminus \RR_{g}(\infty) $ where $ \RR_{g}(\infty) $ is the avoided region defined in Section \ref{sec-Avoided region}. For a given small enough number $ \e > 0 $, there exists the sequence $ \{ b_i \}_{i \in \N_0 } $ defined as  
$$ 
b_{i+1} = g(b_i) 
$$
for each $ i =0,1,2, \ldots , N $ which satisfies that 
\begin{align*}
| a_N - b_N | &\leq \frac{M^N-1}{M-1}\, \e 
\end{align*}
where $ N $ is the uniformly escaping time from the region $ \hat{\C} \setminus S\left(1 + \frac{\tau}{2} \right) $ and $ M = \frac{2(k-1)^4}{k^3 \delta^2} $. 
\end{prop}

\begin{proof}
$ \frac{M}{2} $ is an upper bound of $ |g'| $ in $ \C \setminus \RR_{g}(\infty) $ by Corollary \ref{cor-upper bound of absolute value of g'}. The triangular inequality and Lemma \ref{lem-mean value thoerem for holo functin} implies that 
\begin{align*}
| a_N - b_N | &\leq | a_N - g(a_{N-1}) | + | g(a_{N-1}) - g(b_{N-1}) | + | g(b_{N-1}) - b_N | \\
&\leq \e + M \, | a_{N-1} - b_{N-1} | 
\end{align*}
where $ \displaystyle M \geq \sup_{z \in \C \setminus \RR_{g}(\infty)} 2|g'| $. Observe that if $ \delta > 0 $ is sufficiently small, then $ M > 1 $ in the region $ \hat{\C} \setminus S\left(1 + \frac{\tau}{2} \right) $. Thus we have 
\begin{align*}
| a_N - b_N | + \frac{\e}{M-1} \leq \, M \left( | a_{N-1} - b_{N-1} | + \frac{\e}{M-1} \right) .
\end{align*}
Then $ | a_N - b_N | $ is bounded above by the geometric sequence with rate $ M $. 
\begin{align*}
| a_N - b_N | &\leq M^N\,\left( | a_{0} - b_{0} | + \frac{\e}{M-1} \right) - \frac{\e}{M-1} \\[0.2em]
&= M^N\, | a_{0} - b_{0} | + \frac{M^N-1}{M-1}\, \e
\end{align*}
Hence, if we choose $ b_0 = a_0 $, then 
\begin{align*}
| a_N - b_N | &\leq \frac{M^N-1}{M-1}\, \e .
\end{align*}
\end{proof}

\noindent We show the Hyers-Ulam stability of hyperbolic M\"obius map outside the avoided region.

\begin{thm} \label{thm-stabiliy of hyp Mobius transform}
Let $ g $ be a hyperbolic M\"obius map. For a given $ \e > 0 $, let a complex valued sequence $ \{ a_n \}_{n \in \N_0 } $ satisfies the inequality
$$
| a_{i+1} - g(a_i) | \leq \e
$$
for all $ i \in \N_0 $. Suppose that a given point $ a_0 \in {\C} \setminus \RR_{g}(\infty) $ where $ \RR_{g}(\infty) $ is the avoided region defined in Section \ref{sec-Avoided region}. For a small enough number $ \e > 0 $, there exists the sequence $ \{ b_i \}_{i \in \N_0 } $ 
$$ 
b_{i+1} = g(b_i) 
$$
satisfies that $ |a_i - b_i | \leq H(\e) $ for all $ i \in \N_0 $ for each $ i \in \N $ where the positive number $ H(\e) $ converges to zero as $ \e \rightarrow 0 $. 
\end{thm}

\begin{proof}
Suppose first that $ a_0 \in S\left(1 + \frac{\tau}{2} \right) $. Then by Proposition \ref{prop-stability on S1+t}, we have the inequality
\begin{align} \label{eq-stability first case}
|b_i - a_i| \leq \frac{1 - K^i}{1 - K} \,\e
\end{align}
for some $ K < 1 $. Secondly, assume that $ a_0 \in \left(\hat{\C} \setminus S\left(1 + \frac{\tau}{2} \right) \right) \setminus \RR_{g}(\infty) $ and $ i \leq N $ where $ N $ is the escaping time from the region $ \left(\hat{\C} \setminus S\left(1 + \frac{\tau}{2} \right) \right) \setminus \RR_{g}(\infty) $. Then by Proposition \ref{prop-stability on C minus S}, 
\begin{align} \label{eq-stability second case}
|b_i - a_i| \leq \frac{M^i-1}{M-1}\, \e
\end{align}
where $ M = \frac{2(k-1)^4}{k^3 \delta^2} $. Suppose that $ a_0 \in \left(\hat{\C} \setminus S\left(1 + \frac{\tau}{2} \right) \right) \setminus \RR_{g}(\infty) $ but $ i > N $ for the last case. Then we combine the first and second case as follows
\begin{align} \label{eq-stability third case}
|b_i - a_i| \leq \left( \frac{M^N-1}{M-1} + \frac{1 - K^{i-N}}{1 - K} \right)\e
\end{align}
where $ K $ and $ M $ are the numbers used in the inequality \eqref{eq-stability first case} and \eqref{eq-stability second case}. 
\end{proof}



\end{document}